\documentclass[11pt]{article}
\usepackage{amsmath,amssymb}
\usepackage{color}

\newcommand\eq[1]{(\ref{eq:#1})}

\newtheorem{theorem}{Theorem}
\newtheorem{lemma}{Lemma}

\newtheorem{remark}{Remark}
\newtheorem{proposition}{Proposition}

\newtheorem{corollary}{Corollary}

\def\endpf{{\ \hfill\hbox{\vrule width1.0ex height1.0ex}\parfillskip 0pt
	}}
	\newenvironment{proof}{\noindent{\bf Proof:}}{\endpf}
\begin{document}

\title{Minimizing a stochastic convex function subject to stochastic constraints and some applications}
\author{Royi Jacobovic\thanks{Department of Statistics and data science; The Hebrew University of Jerusalem; Jerusalem 9190501; Israel.
{\tt royi.jacobovic@mail.huji.ac.il, offer.kella@gmail.com}}\and Offer Kella\footnotemark[1] \thanks{Supported by grant No. 1647/17 from the Israel Science Foundation and
the Vigevani Chair in Statistics.}}

\date{February 26, 2020}
		\maketitle
		\begin{abstract}
In the simplest case, we obtain a general solution to a problem of minimizing an integral of a nondecreasing right continuous stochastic process from zero to some nonnegative random variable $\tau$, under the constraints that for some nonnegative random variable $T$, $\tau\in[0,T]$ almost surely and $E\tau=\alpha$ (or $E\tau\le \alpha$) for some $\alpha$. The nondecreasing process and $T$ are allowed to be dependent. In fact a more general setup involving $\sigma$ finite measure, rather than just probability measures is considered and some consequences for families of stochastic processes are given as special cases. Various applications are provided.
 		\end{abstract}

 Quadratic function with random coefficients. Clearing process. Constrained portfolio optimization. Neyman-Pearson lemma.

\bigskip
\noindent {\bf AMS Subject Classification (MSC2010):} 90C15, 60G99.

		\section{Introduction}
		Motivated by the example described in Section \ref{sec: Levy storage}, this work presents an approach to solve a certain kind of stochastic programming problems. For general reviews about stochastic programming problems see, \textit{e.g}., \cite{Rusc2003,prekopa2013}.   Given a general probability space, we were initially motivated by finding an optimal random variable $\tau$ that minimizes $E\int_0^\tau \xi(s)ds$, where $\xi(\cdot)$ is a nondecreasing right continuous stochastic process such that $\xi(t)<\infty$ for every $t\ge 0$, subject to two types of constraints. The first is $P(0\le \tau\le T)=1$ where $T$ is some random variable (possibly infinite) which is not necessarily independent of $\xi(\cdot)$. The second constraint is $E\tau=\alpha$ where $\alpha\in[0,ET]$. It turns out that there is a precise and quite clean representation of the optimal $\tau$ in terms of the pseudo-inverse process associated with $\xi(\cdot)$. In particular, $\varphi(t):=\int_0^t\xi(s)ds$ may be non-differentiable with positive probability. This makes the current work related to non-differentiable convex  optimization. For references regarding deterministic non-differentiable optimization see, \textit{e.g}., \cite{Goffin2002,Lemarechal1989} and Section 11 of  \cite{Zornig2014}. For works about stochastic non-differentiable convex optimization see, \textit{e.g}., \cite{Bertsekas1973,Ermoliev1983,Plambeck1996}. Section \ref{sec:main} includes the main results of this paper. In fact, they are shown for a somewhat more general setup involving $\sigma$-finite measures rather than just probability measures. As to be shown later, this description is useful in various applications to which the other sections are devoted. Section \ref{sec:deterministic} is about the case where $\xi(\cdot)$ is a deterministic function. Section \ref{sec: Levy storage} describes the initial motivation for the current research. It is about a L\'evy-driven storage queue with a controller who picks an output rate in order to minimize the long-run average cost given a certain cost structure. Section \ref{sec:quadratic} refers to a special case when $\xi(\cdot)$ is a strictly increasing linear function with random coefficients. Section \ref{sec:linear} is about the special case where $\xi(\cdot)$ does not depend on $t$. We argue that the results for this case can be applied in hypothesis testing. Section \ref{sec:portfolio} applies the main results in order to solve a problem resulting from the martingale method for solving a dynamic portfolio optimization problem in a continuous time complete market (for details see, \textit{e.g}., Section 3 of \cite{Korn1997}). Section \ref{sec:clearing} is an application of the current method to find an optimal clearing time for a general clearing model with fixed setup cost and nondecreasing holding cost function (not necessarily linear). For some background on clearing processes see (among others) \cite{kella98,kellastadje15,kellayor10,stidham73,stidham77,whitt81}. In Section \ref{sec:renewal} we consider the case where $\xi(\cdot)$ is a renewal counting process and $T$ is independent of $\xi(\cdot)$. When $T$ has an exponential distribution there is a particularly explicit formula for the solution. Section \ref{sec:separable} shows that the current results can be applied to a deterministic setup of separable convex objective function with linear constraints. Finally, Section \ref{sec: regulation} is about price-regulation of an $M/G/1$ queue with customers having nonincreasing stochastic marginal utilities.

\section{The main results}\label{sec:main}
Denote $x^+=x\vee 0$, $x^-=-x\wedge 0$, where $x\vee y=\max(x,y)$, $x\wedge y=\min(x,y)$. Also, for some function $f$, whenever the limits exist, we denote $f(t+)=\lim_{s\downarrow t}f(s)$ and $f(t-)=\lim_{s\uparrow t}f(s)$. As usual, $\mu$-{\em a.s.} abbreviates {\em almost surely} with respect to some (sigma finite or probability) measure $\mu$ and for a sigma finite measure space $(X,\mathcal{X},\mu)$ and $\mathcal{X}$-measurable $\zeta:X\to[-\infty,\infty]$ we denote $\mu\zeta=\int_X \zeta d\mu$ (whenever either $\mu\zeta^+<\infty$ or $\mu\zeta^-<\infty$). For the special case where $\mu$ is a probability measure then we write $E\zeta$ (expected value) instead of $\mu\zeta$. Finally, $\nu\ll\mu$ is for {\em $\nu$ is absolutely continuous with respect to $\mu$}.

From here on when we write $\inf\{t|t\in A\}$ we mean $\inf\{t|t\in A\cap[0,\infty)\}$. When $A\cap(0,\infty)$ is empty, the infimum is defined to be~$\infty$. The following is the main idea that leads to our main result.
\begin{lemma}\label{lem:main}
Let $\varphi:[0,\infty)\to\mathbb{R}$ be convex, right continuous at zero (hence, continuous on $[0,\infty)$), with right derivative $\xi$ (necessarily nondecreasing and right continuous). For $\lambda\in\mathbb{R}$, denote
\begin{equation}
\eta(\lambda)=\inf\{t|\xi(t)\ge \lambda\}\,,
\end{equation}
where $\eta(\lambda)=\infty$ if $\{t|\xi(t)\ge \lambda\}$ is empty.
For a given $T\in[0,\infty]$ (possibly infinite) and $\lambda\in \mathbb{R}$ let
\begin{equation}\label{eq:taulambda}
\tau_\lambda=\eta(\lambda)\wedge T\ .
\end{equation}
If $\tau_\lambda<\infty$, then for every finite $t\in[0,T]$
\begin{equation}\label{eq:lambda}
\varphi(t)\ge \varphi(\tau_\lambda)+\lambda(t-\tau_\lambda)\ .
\end{equation}
Moreover, if $\tau_u<\infty$ for some $u>\lambda$ then also
\begin{equation}\label{eq:lambda+}
\varphi(t)\ge\varphi(\tau_{\lambda+})+\lambda(t-\tau_{\lambda+})\ ,
\end{equation}
where $\tau_{\lambda+}=\eta(\lambda+)\wedge T$.

\end{lemma}

Observe that if we would replace $\lambda(t-\tau_\lambda)$ on the right side of \eq{lambda} by $\xi(\tau_\lambda)(t-\tau_\lambda)$ then the resulting inequality would be an immediate consequence of convexity (since $\xi(\tau_\lambda)$ is a subgradient at $\tau_\lambda$) and would be valid for any choice of $\tau_\lambda$, not necessarily the one defined in \eq{taulambda}. However, for what follows, it is important to have $\lambda$ rather than $\xi(\tau_\lambda)$ as the multiplier.

\medskip
\begin{proof}
It is well known that $\xi(t)\ge \lambda$ if and only if $t\ge \eta(\lambda)$. Therefore, $\varphi(t)-\lambda t=\int_0^t(\xi(s)-\lambda)ds$ is decreasing on $[0,\eta(\lambda))$ (empty when $\eta(\lambda)=0$) and nondecreasing on $[\eta(\lambda),\infty)$ (empty when $\eta(\lambda)=\infty$). Hence, it is decreasing on $[0,\tau_\lambda)$ and nondecreasing on $[\tau_\lambda,T]\cap[\tau_\lambda,\infty)$. This implies that when $\tau_\lambda<\infty$, it minimizes $\varphi(t)-\lambda t$ on $[0,T]$. Thus, for every finite $t\in[0,T]$ and every $\lambda$ such that $\tau_\lambda<\infty$ we have that
\begin{equation}
\varphi(\tau_\lambda)-\lambda\tau_\lambda\le \varphi(t)-\lambda t
\end{equation}
which is equivalent to \eq{lambda}. Clearly, \eq{lambda+} follows from the continuity of $\varphi$.
\end{proof}

We now abuse the notation and instead of a function $\xi$, a nonnegative constant $T$ and a constant $\tau_\lambda$, from here on, these would now become functions of the form $\xi(t)=\xi(x,t)$, where we suppress the (functional) dependence on $x$. Although our main concern is with probability spaces (in which case $\xi$ is a stochastic process), it will prove useful to state the following more general result from which everything else follows. This is the main result of this paper. We will abbreviate $\zeta\in\mathcal{X}$ to mean that $\zeta$ is $\mathcal{X}$-measurable and $\tau\in[0,T]$ $\mu$-a.s. to mean that $\tau(x)\in[0,T(x)]$ for $\mu$-almost all $x\in X$.

\begin{theorem}\label{th:main1}
Given a measurable space $(X,\mathcal{X})$ and sigma finite measures $\nu,\mu$ such that $\nu\ll \mu$, assume that $\xi(t)\in\mathcal{X}$ for each $t\ge 0$, $\xi(t)=\xi(x,t)$ is right continuous and nondecreasing in $t$ for each $x\in X$ and $\mu$-a.s. finite for each $t>0$. Let $T=T(x)\in\mathcal{X}$ be $\mu$-a.s. nonnegative (possibly infinite) satisfying
\begin{equation}
\nu\int_0^T\xi(s)^-ds<\infty\ .
\end{equation}
For $\alpha\in(0,\mu T)$ and $\tau=\tau(x)$, consider
\begin{align}\label{eq:minalpha}
\min\ &\nu\int_0^\tau\xi(s)ds\nonumber\\
\text{s.t.}\ &\tau\in\mathcal{X}\nonumber\\
&\tau\in [0,T]\ \mu\text{-a.s.}\\
&\mu \tau=\alpha\ .\nonumber
\end{align}
Let $Y=d\nu/d\mu$ be a nonnegative finite version of the Radon-Nikodym derivative. With $\tau_\lambda=\inf\{t|Y\xi(t)\ge \lambda\}\wedge T$, if there exists a $\lambda$ satisfying $\mu \tau_{\lambda}=\alpha$ or $\mu \tau_{\lambda+}=\alpha$, then, respectively, $\tau_{\lambda}$ or $\tau_{\lambda+}$ solves \eq{minalpha}.
Otherwise, either $\mu \tau_{\lambda}=\infty$ for all $\lambda\in\mathbb{R}$ or there exists a $\lambda$ for which $\mu \tau_{\lambda}<\alpha<\mu \tau_{\lambda+}$. If $\mu \tau_{\lambda+}<\infty$, let
\begin{equation}\label{eq:p}
q\equiv\frac{\alpha-\mu\tau_{\lambda}}{\mu\tau_{\lambda+}-\mu \tau_{\lambda}}\ .
\end{equation}
Then, $(1-q)\tau_{\lambda}+q\tau_{\lambda+}$ solves \eq{minalpha}.
\end{theorem}
Observe that any version of $Y=\frac{d\nu}{d\mu}$ is $\mu$-a.s. nonnegative and finite. Thus, we can always replace it by $Y1_{\left(0,\infty\right)}(Y)$ to obtain the nonnegative finite version assumed in Theorem \ref{th:main1}. We note that if one prefers that $\xi(t)$ is right continuous and nondecreasing $\mu$-a.s. rather than for every $x\in X$, then in addition one needs to assume in Theorem~\ref{th:main1} that $(X,\mathcal{X},\mu)$ is complete. This is a technical nuisance which we prefer to avoid here.

\medskip

\begin{proof}
	We first observe that the assumption $\nu\int_0^T\xi(s)^-ds<\infty$ is needed in order to insure that $\nu\int_0^\tau\xi(s)ds$ is well defined (possibly infinite) for each $\mathcal{X}$-measurable $\tau\in[0,T]$ $\mu$-a.s.
	
	Next we note that since $\xi(t)$ is nondecreasing and right continuous for each $x\in X$ and is $\mu$-a.s. finite for each $t>0$, then so is $Y\xi(t)$. Thus there is no loss of generality in assuming that $\nu=\mu$ (with $Y=1$).
	
	We recall that from right continuity it follows that as a function of $(x,t)$, $\xi\in\mathcal{X}\otimes\mathcal{B}[0,\infty]$ (jointly measurable, \textit{e.g}., Remark~1.4 on p.~5 of~\cite{KS88}). Here $\mathcal{B}$ is for {\em Borel}. Thus $\int_0^t\xi(s)ds\in\mathcal{X}$ for each $t\ge 0$. Since
	\begin{equation}\label{eq:tauiff}
	\{\tau_\lambda\le t\}=\{\xi(t)\ge \lambda\}\cup\{T\le t\}\ ,
	\end{equation}
	this implies that $\tau_\lambda\in\mathcal{X}$ for each $\lambda$ (and, in fact, that it is jointly measurable as a function of $x,\lambda$, but this will not be needed here).
	
	If for some $\lambda$ either $\mu\tau_\lambda=\alpha$ or $\mu\tau_{\lambda+}=\alpha$ then we simply apply one of the inequalities \eq{lambda},\eq{lambda+} with $t=\tau$ and integrate with respect to $\mu$, observing that $\mu\left[\lambda(\tau-\tau_\lambda\right]=\lambda(\alpha-\alpha)=0$ or $\mu\left[\lambda(\tau-\tau_{\lambda+})\right]=\lambda(\alpha-\alpha)=0$ (since it is required that $\mu\tau=\alpha$).
	
	Recall that $\xi(t)$ is $\mu$-a.s. finite for every $t>0$. From \eq{tauiff} it follows that $\mu$-a.s. $\lim_{\lambda\to-\infty}\tau_\lambda=0$ and $\lim_{\lambda\to\infty}\tau_\lambda=T$.
	Thus, from monotone convergence $\mu\tau_\lambda\to\mu T$ as $\lambda\to\infty$ and, when $\mu\tau_\lambda<\infty$ for some $\lambda\in\mathbb{R}$, it converges by dominated convergence to $0$ as $\lambda\to-\infty$. Thus, when $\mu\tau_\lambda$ is not infinite for all $\lambda$, for each $\alpha\in(0,\mu T)$ such that there is no $\lambda$ for which $\mu\tau_\lambda=\alpha$ or $\mu\tau_{\lambda+}=\alpha$, we can take $\lambda$ such that
	\begin{equation}
	\mu\tau_{\lambda}<\alpha<\mu\tau_{\lambda+}\ .
	\end{equation}
	Assuming that the right side is finite (equivalent to $\mu\tau_u<\infty$ for some $u>\lambda$), then clearly, $(1-q)\tau_{\lambda}+q\tau_{\lambda+}\in[0,T]$ $\mu$-a.s. and $\mu\left[(1-q)\tau_{\lambda}+q\tau_{\lambda+}\right]=\alpha$. From \eq{lambda}, \eq{lambda+} and the convexity of $\int_0^t\xi(s)ds$ in $t$, we have that
	\begin{align}\label{eq:13}
	\mu\int_0^{(1-q)\tau_{\lambda}+q\tau_{\lambda+}}\xi(s)ds&\le(1-q)\mu\int_0^{\tau_{\lambda}}\xi(s)ds+qE\int_0^{\tau_{\lambda+}}\xi(s)ds\nonumber\\
	&\le (1-q)\left(\mu\int_0^\tau\xi(s)ds-\lambda(\alpha-\mu\tau_{\lambda})\right)\\
	&\quad+q\left(\mu\int_0^\tau\xi(s)ds-\lambda(\alpha-\mu\tau_{\lambda+})\right)\nonumber\\
	&=\mu\int_0^\tau\xi(s)ds\ .\nonumber
	\end{align}for every $\tau$ satisfying the constraints, so the proof is complete.
\end{proof}

We observe that when $\alpha=0$, every $\mathcal{X}$-measurable $\tau\in[0,T]$ $\mu$-a.s. with $\mu\tau=0$ necessarily satisfies that $\tau=0$ $\mu$-a.s. Similarly, when $\alpha=\mu T<\infty$, every $\mathcal{X}$-measurable $\tau\in[0,T]$ $\mu$-a.s. with $\mu\tau=\mu T$ necessarily satisfies that $\tau=T$ $\mu$-a.s. Also note that for $\alpha\not\in[0,ET]$ the problem is infeasible. Thus, these cases are trivial.

We note that when $\xi(t,x)$ is strictly increasing in $t$ for every $x$, then $\tau_\lambda$ is continuous in $\lambda$. Hence, if $\mu\tau_\lambda<\infty$ for all $\lambda$ (\textit{e.g}., when $\mu T<\infty$), then $\mu\tau_\lambda$ is continuous in $\lambda$ and for each $\alpha\in(0,\mu T)$ there is a $\lambda\in\mathbb{R}$ for which $\mu\tau_\lambda=\alpha$. Therefore, in this case there is no need to take a convex combination of $\tau_\lambda$ and $\tau_{\lambda+}$.

\begin{proposition}\label{prop:1}
When, in addition to the assumptions of Theorem~\ref{th:main1}, $\mu\tau_\lambda<\infty$ for all $\lambda\in\mathbb{R}$ (\textit{e.g}. when $\mu T<\infty$), with $\tau(\alpha)$ denoting the optimum of \eq{minalpha} (clearly nondecreasing in $\alpha$), $f(\alpha)=\nu\int_0^{\tau(\alpha)}\xi(s)ds$ is a convex function of $\alpha$ on $[0,\mu T]\cap[0,\infty)$. Moreover,
$\lim_{\alpha\uparrow \mu T}f(\alpha)=\nu\int_0^T\xi(s)ds$ (including the case that $T$ is not $\mu$-a.s. finite) and if $\nu\int_0^{\tau_\lambda}\xi(s)^+ds<\infty$ for some $\lambda$ then $\lim_{\alpha\downarrow0}f(\alpha)=0$.
\end{proposition}

\begin{proof}
If we take $u\in(0,1)$ and some finite $\alpha,\beta\in[0,ET]$, then
\begin{equation}
\mu((1-u)\tau(\alpha)+u\tau(\beta))=(1-u)\alpha+u\beta
\end{equation}
so that by minimality of $\tau((1-u)\alpha+u\beta)$ and convexity of $\int_0^t\xi(s)ds$ in $t$ we have that
\begin{align}
\nu\int_0^{\tau((1-u)\alpha+u\beta)}\xi(s)ds&\le \nu\int_0^{(1-u)\tau(\alpha)+u\tau(\beta)}\xi(s)ds\nonumber\\
&\le(1-u)\nu\int_0^{\tau(\alpha)}\xi(s)ds+u \nu\int_0^{\tau(\beta)}\xi(s)ds\ .
\end{align}
Now, we recall (see the proof of Theorem~\ref{th:main1}) that $\lim_{\lambda\to-\infty}\tau_\lambda=0$ and $\lim_{\lambda\to\infty}\tau_\lambda=T$. This implies both $\int_0^{\tau_\lambda}\xi(s)ds$ converges to zero as $\lambda\to-\infty$ and to $\int_0^T\xi(s)ds$ as $\lambda\to\infty$. Dominated (for $\lambda\to-\infty$) and monotone (for $\lambda\to\infty$) convergence (separately for $\xi(s)^+$ and $\xi(s)^-$) implies that this also holds for the integral with respect to $\nu$. Recall that we assume that $\nu\int_0^T\xi(s)^-ds<\infty$. Also note that since $\tau(\alpha)$ is nondecreasing in $\alpha$, then $\nu\int_0^{\tau(\alpha)}\xi(s)^\pm ds$ are nondecreasing in $\alpha$. Thus they have a limit as $\alpha$ converges to zero or to $\mu T$ (which for the latter, with $\xi(s)^+$, could be infinite). Thus if we take $\alpha(\lambda)=\mu\tau_\lambda$ then the same limits are obtained when $\lambda\to\pm\infty$.
\end{proof}

We observe that if instead the constraint $\tau\in[0,T]$ a.s. we take $\tau\in[S,T]$ a.s. where $S\in\mathcal{X}$ satisfies $\mu S<\infty$ and $0\le S\le T$ $\mu$-a.s., then upon taking $\tilde\xi(t)=\xi(S+t)$, $\tilde T=T-S$ and $\tilde\alpha=\alpha-\mu S$, we are back to the original setup. Therefore, Theorem~\ref{th:main1} gives a solution for this case as well. Note that for this optimization problem we may take $\xi(\cdot)$ to be indexed by $\mathbb{R}$ on and there is no need to assume that $S,T$ are nonnegative.

Finally we also observe that if $\mu$ and $\nu$ are equivalent measures, then the problem \eq{minalpha} may be replaced by a problem in which $\mu=\nu$, but the equality $\mu\tau=\alpha$ is replaced by $\mu A\tau=\alpha$ where $A$ is strictly positive (and finite). This implies the following two corollaries for two special cases. The first is when $\mu$ is replaced by a product measure associated with a counting measure and a probability measure and the second is where $\mu$ is replaced by the product of Lebesgue measure and a probability measure. It will be useful to refer to those in the examples that will appear later. The straightforward proofs are omitted.

\begin{corollary}\label{cor:main2}
Given a probability space $(\Omega,\mathcal{F},P)$, assume that $\{\xi_i(t)|t\ge 0\}$, are nondecreasing right continuous stochastic processes with $P(|\xi_i(t)|<\infty)=1$ for all $t>0$ and $i$ in some finite or countable index set. Let $T_i$ be nonnegative (possibly infinite) random variables satisfying
\begin{equation}
\sum_iE\int_0^{T_i}\xi_i(s)^-ds<\infty
\end{equation}
and let $A_i$ be positive and finite random variables.
Consider the following stochastic optimization problem for $\alpha\in \left(0,\sum_iET_i\right)$.
\begin{align}\label{eq:multi}
\min\ &\sum_i E\int_0^{\tau_i}\xi_i(s)ds\nonumber\\
s.t.\ &\tau_i\in\mathcal{F}\ ,\ \forall i \\&\tau_i\in[0,T_i]\ P\text{-a.s.},\ \forall i \nonumber\\
&\sum_i EA_i\tau_i=\alpha\nonumber
\end{align}
Denote $\tau_{i,\lambda}=\inf\{t|\xi_i(t)\ge A_i\lambda\}\wedge T_i$. If there exists a $\lambda$ satisfying $\sum_iEA_i\tau_{i,\lambda}=\alpha$ or $\sum_iEA_i\tau_{i,\lambda+}=\alpha$, then, respectively, $\tau_{i,\lambda}$ or $\tau_{i,\lambda+}$ for all $i$, solve \eq{multi}. Otherwise, either $\sum_iEA_i\tau_{i,\lambda}=\infty$ for all $\lambda\in\mathbb{R}$ or there exists a $\lambda$ for which $\sum_iEA_i\tau_{i,\lambda}<\alpha<\sum_i EA_i\tau_{i,\lambda+}$. If $\sum_i EA_i\tau_{i,\lambda+}<\infty$, let
\begin{equation}\label{eq:qmulti}
q=\frac{\alpha-\sum_iEA_i\tau_{i,\lambda}}{\sum_i EA_i(\tau_{i,\lambda+}-\tau_{i,\lambda})}\ .
\end{equation}
Then, $(1-q)\tau_{i,\lambda}+q\tau_{i,\lambda+}$, for $1\le i\le n$, solves \eq{multi}. \end{corollary}

\begin{corollary}\label{cor:main3}
Given a probability space $(\Omega,\mathcal{F},P)$ and denoting the Lebesgue measure by $m$ and $ds=m(ds)$, assume that $\xi(s,t)=\xi(\omega,s,t)$ is a measurable process as a function of $(\omega,s)$ for each fixed $t\ge 0$ and nondecreasing right continuous in $t$ for each fixed $(\omega,s)\in \Omega\times [0,\infty)$, with $P(|\xi(s,t)|<\infty)=1$ for all $t>0$ and $m$-almost each $s\in [0,\infty)$. Let $T_s=T_s(\omega)\in\mathcal{B}[0,\infty)\otimes\mathcal{F}$ with $P(T_s\ge 0)=1$ for $m$-almost all $s\in [0,\infty)$ and assume that
\begin{equation}
\int_0^\infty E\int_0^{T_s}\xi(s,t)^-dt\,ds<\infty\ .
\end{equation}
Finally let $A_s=A_s(\omega)$ be a measurable process and $P(0<A_s<\infty)=1$ for $m$-almost all $s\in[0,\infty)$.
Consider the following optimization problem.
\begin{align}\label{eq:cont}
\min\ &\int_0^\infty E\int_0^{\tau_s}\xi(s,t)dt\,ds\nonumber\\
\text{s.t.}\ &\tau_s \text{ is a measurable process}\\
&P(0\le \tau_s\le T_s)=1\ \text{for }m\text{-almost all\ } s\in [0,\infty)\nonumber\\
&\int_0^\infty EA_s\tau_s\,ds=\alpha\ .\nonumber
\end{align}
Let $\tau_{s,\lambda}=\inf\{t|\xi(s,t)\ge A_s\lambda\}\wedge T_s$. If there exists $\lambda\in \mathbb{R}$ such that either $\int_0^\infty EA_s\tau_{s,\lambda}\,ds=\alpha$ or $\int_0^\infty EA_s\tau_{s,\lambda+}\,ds=\alpha$ then, respectively, $\tau_{s,\lambda}$ or $\tau_{s,\lambda+}$, for $s\in[0,\infty)$, solve \eq{cont}. Otherwise, either $\int_0^\infty EA_s\tau_{s,\lambda}\,ds=\infty$ for all $\lambda$ or there exists some $\lambda$ for which $\int_0^\infty EA_s\tau_{s,\lambda}\,ds<\alpha<\int_0^\infty EA_s\tau_{s,\lambda+}\,ds$. When $\int_0^\infty EA_s\tau_{s,\lambda+}\,ds<\infty$, denote
\begin{equation}
q=\frac{\alpha-\int_0^\infty EA_s\tau_{s,\lambda}\,ds}
{\int_0^\infty EA_s(\tau_{s,\lambda+}-\tau_{s,\lambda})\,ds}
\end{equation}
and then $(1-q)\tau_{s,\lambda}+q\tau_{s,\lambda+}$, for $s\in[0,\infty)$, solve \eq{cont}.
\end{corollary}

\begin{remark}\label{rem:mixed}{\rm
It should be observed that upon taking $\tau=(1-q)\tau_\lambda+q\tau_{\lambda+}$ in \eq{13} it follows that the first inequality in \eq{13} is actually an equality. Therefore we can take a probability space $([0,1],\mathcal{B}([0,1]),m)$ (where $m$ is Lebesgue measure) and consider the random variable $I=1_{[0,q]}(\omega)$. Then take $(1-I)\tau_\lambda+I\tau_{\lambda+}$ on the space $([0,1]\otimes X,\mathcal{B}([0,1])\otimes\mathcal{X},m\otimes \mu)$ and obtain that
\begin{align}
m\otimes \mu\int_0^{(1-I)\tau_\lambda+I\tau_{\lambda+}}\xi(s)&=(1-q)\mu \int_0^{\tau_\lambda}\xi(s)ds+q\int_0^{\tau_{\lambda+}}\xi(s)ds\nonumber\\
&\le \mu\int_0^\tau\xi(s)ds= m\otimes\mu\int_0^\tau\xi(s)ds
\end{align}
for every $\tau$ satisfying the constraints. In this case we may refer to $(1-I)\tau_\lambda+I\tau_{\lambda+}$ as a {\em mixed} strategy.
}
\end{remark}

\begin{remark}\label{rem:discrete}{\rm
It is easy to check that if $\xi$ is indexed by $\mathbb{Z}_+$ (instead of $[0,\infty)$) and $\tau$ and $T$ are integer valued, then the results of this section continue to hold without change with the exception that $\tau_\lambda$ is defined to be $\inf\{n|\xi(n)\ge \lambda\}\wedge T$ (integer valued) and that instead of $(1-q)\tau_\lambda+q\tau_{\lambda+}$ (which is not necessarily an integer) we need to take a mixed strategy $(1-I)\tau_\lambda+I\tau_{\lambda+}$ as appearing in Remark~\ref{rem:mixed}. Similarly, the same is true for $\tau_{i,\lambda}$ and $\tau_{s,\lambda}$.
}
\end{remark}

\begin{remark}\label{rem:le}{\rm
Observe that if we replace $\mu\tau=\alpha$ in \eq{minalpha} by $\mu \tau\le \alpha$, then from Lemma~\ref{lem:main} it follows that, when finite, $\tau_0$ and $\tau_{0+}$ (as defined in Theorem~\ref{th:main1}), minimize $\int_0^t Y\xi(s)ds$ on $[0,T]\cap[0,\infty)$ for each $x$ and therefore it minimizes the integral with respect to $\mu$. Thus, if $\mu A\tau_0\le \alpha$ then $\tau_0$ is an optimal solution for this modified problem. Otherwise, the optimal solution is the one given in Theorem~\ref{th:main1}. The reason is that $\lambda$ for which $E\tau_\lambda\le \alpha\le E\tau_{\lambda+}$ is necessarily negative and thus replacing $\alpha$ by $\mu\tau\in[0,\alpha]$ in \eq{13} gives on the right side $\mu\int_0^\tau\xi(s)ds-\lambda(\mu\tau-\alpha)$. Since $\lambda<0$ it follows that $-\lambda(\mu\tau-\alpha)\le 0$. This can also be deduced from the convexity reported in Proposition~\ref{prop:1}.
Naturally, the same is valid for Corollaries~\ref{cor:main2} and~\ref{cor:main3}.}
\end{remark}

We now proceed to some examples.

\section{Minimizing a deterministic convex function}\label{sec:deterministic}
When $\xi$ is deterministic we can conclude the following.
\begin{corollary}\label{cor:det}
Assume that $\psi:[0,\infty)\to\mathbb{R}$ is strictly convex, right continuous at zero (deterministic) and $T$ is a nonnegative, finite mean random variable with distribution $F$. Denote
\[
F_e(t)=\frac{1}{ET}\int_0^t(1-F(s))ds
\]
(stationary remaining lifetime distribution). Then for every $p\in(0,1)$
\begin{equation}\label{eq:taup}
\tau_p=F_e^{-1}(p)\wedge T
\end{equation}
minimizes
\begin{align*}
\min\ &E\psi(\tau)\nonumber\\
\text{s.t.}\ &\tau\in [0,T]\ \text{a.s.}\\
&E\tau=p ET\ .\nonumber
\end{align*}
\end{corollary}

\begin{proof}
Since $\psi$ is strictly convex, then its right derivative $\xi$ is strictly increasing and thus $\eta$ is continuous. Thus, there exists $\lambda$ for which $E\eta(\lambda)\wedge T=pET$ and, by Theorem~\ref{th:main1} the optimal solution is $\eta(\lambda)\wedge T$. Now, since $\eta(\lambda)$ is a deterministic constant, then
\begin{equation}
pET=E\eta(\lambda)\wedge T=\int_0^{\eta(\lambda)}(1-F(s))ds=ETF_e(\eta(\lambda))
\end{equation}
from which it follows that $\eta(\lambda)=F_e^{-1}(p)$, where we note that the inverse is well defined since $F_e$ is strictly increasing and continuous on
\begin{equation}
\left[0,\sup\{t|F(t)<1\}\right)\ .
\end{equation}
\end{proof}

We note that when $\psi$ is convex but not strictly convex and for some bounded below strictly convex function $\varphi$ on $[0,\infty)$ we have that $E\varphi(T)<\infty$, then for any $\tau$ satisfying $\tau\in[0,T]$ a.s. and $E\tau=pET$ we have that $E\varphi(\tau)\le \varphi(0)\vee E\varphi(T)<\infty$ and (since $\psi+\epsilon \varphi$ is strictly convex)
\begin{equation}
E\psi\left(F_e^{-1}(p)\wedge T\right)+\epsilon E\varphi\left(F_e^{-1}(p)\wedge T\right)\le E\psi(\tau)+\epsilon E\varphi(\tau)\ .
\end{equation}
Upon letting $\epsilon\downarrow0$ it follows that with the added condition that $E\varphi(T)<\infty$ for some strictly convex function on $[0,\infty)$, Corollary~\ref{cor:det} is valid for any convex function $\psi$.

One immediate special case is  minimizing $\text{Var}(\tau)$ subject to the constraints $\tau\in[0,T]$ almost surely and $E\tau=pET$ for $p\in(0,1)$. We also note that when $T$ is constant then it is easy to check that $F_e^{-1}(p)=pT$ and thus $\tau_p=pT$, as expected. This, of course, also follows from Jensen's inequality as $\psi(\tau_p)=\psi(E\tau)\le E\psi(\tau)$ for any $\tau$ with $E\tau=pET$. In contrast, we recall that for this case it is well known that the maximum is given by $IT$ where $I\sim\text{Bern}(p)$. To refresh one's memory, this follows from
\begin{equation}
\psi(\tau)\le \left(1-\frac{\tau}{T}\right)\psi(0)+\frac{\tau}{T}\psi(T)
\end{equation}
and then taking expected values, noting that $E\tau/T=p$.

Finally, it is interesting to note that the optimal solution in this section does not depend on the choice of the convex function $\psi$. This is not necessarily so when $\psi$ is stochastic. However, recalling that, for any convex $\psi$, an optimal solution to $\min\sum_{i=1}^n \psi(x_i)$ subject to the constraints $x_i\ge 0$ and $\sum_{i=1}^nx_i=\alpha$ is $x_i=\alpha/n$ (which also does not depend on $\psi$), then perhaps we should not be too surprised.

\section{Output rate control in a L\'evy driven storage system}\label{sec: Levy storage}
Consider a regenerative storage process with a nondecreasing Le\'vy
input (subordinator) such that every cycle may be split into two periods. In the first (off) the output is shut off and the workload accumulates. Consider the following cost structure. A constant holding cost per one unit of workload per one unit of time, a constant setup cost for every cycle, a constant output capacity cost rate. For further explanations regarding these costs see \cite{JK19}. A controller who observes the workload level at the beginning of every on period wants to pick an output rate (which may be different from cycle to cycle) in order to minimize the long-run average cost. As explained in \cite{JK19} this leads the the following optimization problem
\begin{equation} \label{optimization: output rate}
\begin{aligned}
& \min:
& & \frac{K_1+K_2EX+hE\left(\frac{V}{2}X+\frac{\mu\rho}{V} X^2\right)}{K_3+EX} \\
& \text{ s.t.}
& & X\text{ is a random variable} \\ & & & X\geq0, \ \ P\text{-a.s.}
\end{aligned}
\end{equation}
where $K_1,K_2,K_3,\mu,\rho,h$ are determined by the parameters of the model and $V$ is the workload level at the beginning of an on period under the assumption $EV^2<\infty$. A two-phase method is applied to solve \eqref{optimization: output rate}. In Phase I an additional constraint $EX=\alpha$ is imposed, leading to
\begin{equation}
\begin{aligned}
& \min:
& & E\left(\frac{V}{2}X+\frac{\mu\rho}{V} X^2\right) \\
& \text{ s.t.}
& & X\text{ is a random variable}\,,\\ & & &EX=\alpha\,, \\ & & & X\geq0, \ \ P\text{-a.s.}
\end{aligned}
\end{equation}
This is a special case of the setup in Section 2. For further details see \cite{JK19}.

\section{Optimizing a quadratic function with random coefficients}\label{sec:quadratic}
Keeping with the same guideline of the example of Section \ref{sec: Levy storage}, consider the problem
\begin{align}
	\min\ &E\left(A\tau^2+B\tau+C\right)\nonumber\\
	\text{s.t.}\ &\tau\text{ is a random variable}\nonumber\\
	&\tau\in[0,T]\ \text{a.s.}\\
	&ED\tau=\alpha \nonumber
	\end{align}
	for any a.s. finite random variables $A,B,C,D,T$ with $ET<\infty$, $EB^-<\infty$, $EC^-<\infty$ and $P(A>0)=P(D>0)=1$ (having an arbitrary joint distribution). Then, the assumptions of Corollary~\ref{cor:main2} are met with $n=1$, $\xi(t)=2At+B$ for every $t\geq0$ and
\begin{align}
\tau_{\lambda}=\frac{(D\lambda-B)^+}{2A}\wedge T=\frac{(\lambda-B/D)^+}{2A/D}\wedge T\ .
\end{align}
Thus, for this case, if $ET<\infty$ then, by continuity (and dominated convergence), for every $\alpha\in (0,EDT)$ there always is a (finite) $\lambda$ such that $ED\tau_{\lambda}=\alpha$. We also recall that for $\alpha=0$ the solution is a.s. zero and for $\alpha=ET$ it is a.s. $T$.

We also note that the special case of a uniform finite probability space results in the (deterministic) quadratic program and its solution reported in \cite{HKL80}.

\section{Optimizing a linear function with random coefficients}\label{sec:linear}
When the goal is to solve the following problem
\begin{align}
\min\ &E\left(A\tau+B\right)\nonumber\\
\text{s.t.}\ &\tau\in[0,T]\ \text{a.s.}\\
&EC\tau=\alpha \nonumber
\end{align}
where $P(A>0)=P(C>0)=1$ and $EB^-<\infty$, we simply take $\xi(t)=A$ for every $t\ge 0$. For this case we have that $\xi(t)=A$ and thus
\begin{equation}
\tau_\lambda=\inf\{t|A\ge C\lambda\}\wedge T=T1_{\{A/C<\lambda\}}\ .
\end{equation}
As in the quadratic case, $A,B,C,T$ may have an arbitrary joint distribution.

Note that the same holds in the discrete time case, where we recall Remark~\ref{rem:discrete}. In particular, if we take $T=1$ this results in an alternative (but, admittedly, somewhat less elegant) proof of the Neyman-Pearson Lemma or, more generally, uniformly most powerful tests for this setup, where we would like to test the hypotheses
\begin{equation}
\begin{cases}
H_0: P&=P_0\\
H_1: P&=P_1\ ,
\end{cases}
\end{equation}
where $P_0,P_1$ are absolutely continuous with respect to a common $\sigma$-finite measure, under either of the constrains $P_0\tau=\alpha$ or $P_0\tau\le \alpha$. See \cite{Cvitanic2001} for an in-depth treatise of such (and more general) problems which exploits convexity. In particular, compare equations (1.8)-(1.12) there to what appears here.

\section{Relation to portfolio selection} \label{sec:portfolio}
In this section we refer to the classical model of dynamic utility maximization with consumption in a complete continuous-time stock market which is presented in Section 3 of \cite{Korn1997} (see also Section 3 of \cite{KS1998}). In general, the model is about an agent who must dynamically decide how to manage a trade-off between consumption over time and terminal wealth. The model has a fixed finite horizon $T\in(0,\infty)$ and the agent's initial wealth is given by a fixed parameter of the model $x_0\in(0,\infty)$. Furthermore, the agent's utility from consumption of $c\geq0$ units at time $t\geq0$ is given by $U_1(t,c)$. In addition, the agent's utility from a terminal wealth of $x\geq0$ is $U_2(x)$. In addition, $U_1$ and $U_2$ are deterministic functions such that for every $c\geq0$, $U_1(\cdot,c)$ is continuous and for every $t\in[0,T]$, $U_1(t,\cdot)$ and $U_2(\cdot)$ are assumed to be continuously differentiable strictly concave functions with derivatives that satisfy some additional conditions. The exact model description including the stochastic modelling of prices appears in the above-mentioned references. In these references it is shown that one approach to solving the dynamic problem is to first solve the following dual problem (in the literature it is often referred as a  \textit{static} problem \textit{\textit{e.g}.} in \cite{Korn1997} see the last paragraph of page 37 or the first paragraph of Subsection 3.4)

\begin{equation} \label{optimization: portfolio}
\begin{aligned}
& \underset{\{c_t;t\in[0,T]\},X}{\text{maximize}}
& & E\left[\int_0^T U_1(t,c_t)dt+U_2(X)\right] \\
& \text{subject to}
& & X\geq0\ ,\ c_t\geq0\ , \forall t\in[0,T] \ \ , \ \ P-a.s. \\  & & & E\left[\int_0^TH_tc_tdt+H_TX\right]\leq x_0
\end{aligned}
\end{equation}
where $T,x_0>0$ are constants, $\{H_t;t\in[0,T]\}$ is a certain $P$-a.s. positive martingale with respect to a certain filtration $\mathbb{F}$ which represents the  information flow to the agent. The distribution of this martingale is determined endogenously by the model setup and defined at the beginning of Section 2.3 of \cite{Korn1997}. Importantly, this process is not influenced by the decision variables of the optimization. In addition, we also consider the case where either $U_1$ or $U_2$ is identically zero. With these assumptions, as mentioned by \cite{JXZ08}, even when $U_1$ is identically zero, Lagrange multipliers for this problem may not exist and hence the Lagrange method is not always applicable. Another solution which is based on the exact definition of the process $\{H_t;t\geq0\}$ is provided in Section 3.4 of \cite{Korn1997}. The solution of \eqref{optimization: portfolio} can be obtained by our Corollary~\ref{cor:main3}. In particular, one does not need to assume any differentiability assumptions on the utility functions, as is usually assumed in this literature. Moreover,  since $U_1$ and $U_2$ are deterministic, then it can be seen that the solution of \eqref{optimization: portfolio} which is specified by Corollary \ref{cor:main3} is adapted to the filtration generated by $H(\cdot)$ and hence also adapted with respect to $\mathbb{F}$.

Furthermore, the assumption that the utility functions are deterministic can also be relaxed. Namely, take $U_2(\omega,x)$ to be $P$-a.s. concave in $x$ for each $\omega$. In such a case the agent has random utility from terminal wealth. For further details about models with random utilities see, \textit{e.g.}, \cite{Fishburn1998}.  For example, one can think about models with agents whose preferences are determined by a random variable which denotes the agents' types. Now, the assumption is that for every type the utility is convex w.r.t the terminal wealth. The same can be done with respect to $U_1$. Now, let $S$ and $V$ be two nonnegative random variables. In addition assume that $\{\Gamma_t;t\in[0,T]\}$ and $\{\Upsilon_t;t\in[0,T]\}$ are two nonnegative stochastic processes. In particular, assume that these random quantities are exogenous to the model, \textit{\textit{i.e}}. they are not influenced by the choice of $X$ and $\{c_t;t\geq0\}$. The requirement that $X$ and $\{c_t;t\geq0\}$ are nonnegative could be replaced by the constraint
\begin{equation*}
X\in[S,V]\ \ , c_t\in[\Gamma_t,\Upsilon_t],\forall t\in[0,T]\ \ , \ \ P-a.s.
\end{equation*}
to which the results of Section~\ref{sec:main} still apply.  Examples of models considering such constraints are, \textit{e.g}., \cite{Jihan2017,Koo2016,Korn2005,Lanker2006}. Another case which is also covered by the current work is when $T=\infty$ and $U_2$ is identically zero.  Finally, note that this kind of optimizations is also motivated by discrete time models (see, \textit{e.g}., Section 3.3 of \cite{Follmer2011}).

\section{Optimal clearing times in a regenerative clearing process}\label{sec:clearing}
As described by \cite{stidham73}, ``a stochastic clearing system is characterized by a non-decreasing stochastic input process , where $Y(t)$ is the cumulative quantity entering the system in $[0,t]$, and an output mechanism that intermittently and instantaneously clears the system, that is, removes all the quantity currently present." In particular, the clearing system is regenerative if the workload which is associated with this system is a regenerative process. Such systems have been extensively studied in the literature, \textit{e.g}., see \cite{JK2019,kellastadje15,Sigman1993}.

Now, if $\xi$ is a nonnegative process then we can think of $\{\xi(t)|0\le t<\tau\}$ as the first cycle of a (regenerative) clearing process. When $E\tau<\infty$, for such a clearing process an ergodic distribution exists and if $\xi^*$ has this distribution then we have that for any nonnegative Borel $g$,
\begin{equation}
Eg(\xi^*)=\frac{1}{E\tau}E\int_0^\tau g(\xi(s))ds\ .
\end{equation}
Note that if $g$ is a nonnegative, nondecreasing and right continuous function then $g(\xi(\cdot))$ is a nonnegative, nondecreasing right continuous process and we can apply the results of Section~\ref{sec:main} to optimize $Eg(\xi^*)$ subject to the constraints in \eq{minalpha}. This also provides a method for solving the following optimization problem for any given $K> 0$ and nonnegative, nondecreasing right continuous $g$:
\begin{align}
\min\ &\frac{K+E\int_0^\tau g(\xi(s))ds}{E\tau}\nonumber\\ 
\text{s.t.}\ &\tau\in (0,T]\ \text{a.s.}
\end{align}
In this case the cost structure is a setup cost $K$ incurred right after each clearing and a (possibly nonlinear) holding cost function $g$.
The solution is obtained by first restricting the minimization to feasible $\tau$'s satisfying $E\tau=\alpha\in(0,ET]$ to obtain
\begin{equation}\label{eq:h}
h(\alpha)=\frac{K+E\int_0^{\tau(\alpha)}g(\xi(s))ds}{\alpha}
\end{equation}
where $\tau(\alpha)$ denotes the optimal solution from Theorem~\ref{th:main1} and then $h$ is minimized over $(0,ET]$ either analytically, when possible, or numerically.

Finally we observe that for every $a,b\ge0$ such that $a+b>0$ and every $x,y$ we have that
\begin{equation}
\frac{ax+by}{a+b}\ge x\wedge y\,.
\end{equation}
With the notations from \eq{p}, setting
\begin{align}
a&=(1-q)E\tau_{\lambda}\nonumber\\
b&=q E\tau_{\lambda+}\\
x&=\frac{K+E\int_0^{\tau_{\lambda}}g(\xi(s))ds}{E\tau_{\lambda}}\nonumber\\
y&=\frac{K+E\int_0^{\tau_{\lambda+}}g(\xi(s))ds}{E\tau_{\lambda+}}\,,\nonumber
\end{align}
we infer that
\begin{equation}
\frac{K+E\int_0^{\tau(\alpha)}g(\xi(s))ds}{\alpha}\ge \frac{K+E\int_0^{\tau_{\lambda}}g(\xi(s))ds}{E\tau_\lambda}
\wedge \frac{K+E\int_0^{\tau_{\lambda+}}g(\xi(s))ds}{E\tau_{\lambda+}}\ .
\end{equation}
This implies (recall \eq{h}) that
\begin{align}\label{eq:inf}
\inf_{{\tau\in(0,T] \atop \text{a.s.}}}
\frac{K+E\int_0^\tau g(\xi(s))ds}{E\tau}&=
\inf_{\alpha\in(0,ET]}h(\alpha)\nonumber\\
&=\inf_{\lambda|E\tau_\lambda>0}\frac{K+E\int_0^{\tau_\lambda}g(\xi(s))ds}{E\tau_\lambda}\ .
\end{align}
Therefore, for this optimization problem it suffices to restrict attention to random times of the form $\tau_\lambda$, for $\lambda$ such that  $E\tau_\lambda>0$. Note that $E\tau_\lambda>0$ if and only if $P(\tau_\lambda>0)>0$. Since $\eta(\lambda)>0$ if and only if $\lambda>g(\xi(0))$, we can replace `$\lambda|E\tau_\lambda>0$' on the right hand side of \eq{inf} by `$\lambda|P(g(\xi(0))<\lambda,T>0)>0$'. In particular, when $\xi(0)$ independent of $T$ (in the literature it is usually assume to be zero, so that this independence is automatic), then this results in `$\lambda|\lambda>g(\xi(0))$'.

\section{Example: Renewal counting process $\xi$ with independent $T\sim\exp(\mu)$ and a bit more}\label{sec:renewal}
Assume that $T\sim\exp(\theta)$ is independent of $\{\xi(t);t\geq0\}$ which is a renewal counting process with inter-renewal times distributed like some $X$. As usual, it is assumed that $P(X\ge 0)=1$ and $P(X=0)<1$. Since $\xi(0)=0$, then $\eta(\lambda)=0$ for every $\lambda\leq0$ and, for every $\lambda\geq0$,
\begin{align}
\eta(\lambda)&=\inf\{t\geq0;\xi(t)\geq\lambda\}\\&=\inf\{t\geq0;\xi(t)\geq\lceil\lambda\rceil\}=S_{\lceil\lambda\rceil}\nonumber
\end{align}
where $S_{\lceil\lambda\rceil}$ is the $\lceil\lambda\rceil$th renewal time. Therefore, for every $\lambda\geq0$, $\tau_\lambda=T\wedge S_{\lceil\lambda\rceil}$, so that
\begin{align}
E\tau_\lambda &=T\wedge S_{\lceil\lambda\rceil}=E\int_0^{S_{\lceil\lambda\rceil}}e^{-\theta t}dt=\frac{1}{\theta}\left(1-Ee^{-\theta S_{\lceil\lambda\rceil}}\right)\nonumber\\ \\
&=\frac{1}{\theta}\left[1-\left(Ee^{-\theta X}\right)^{\lceil\lambda\rceil}\right]\ .\nonumber
\end{align}
For integer valued $\lambda$ we have that $\tau_{\lambda+}=\tau_{\lambda+1}$ and otherwise $\tau_{\lambda+}=\tau_\lambda$.
It is easily verified that with
\begin{equation}
\lambda_\alpha=\left\lfloor\frac{\log\left(1-\theta\alpha\right)}{\log Ee^{-\theta X}}\right\rfloor\,,
\end{equation}
for $\alpha\in(0,\theta^{-1})$, we either have that
$E\tau_{\lambda_\alpha}=\alpha$ or $E\tau_{\lambda_\alpha+}=E\tau_{\lambda_\alpha+1}=\alpha$ or $E\tau_{\lambda_\alpha}<\alpha< E\tau_{\lambda_\alpha+1}$, in which case the optimal solution is
$(1-q)\tau_{\lambda_\alpha}+q\tau_{\lambda_\alpha+1}$ where $q$ is given by \eq{p}.

If $T$ has a finite mean and is independent of $\xi$ but does not have an exponential distribution, then
\begin{equation}
E\tau_\lambda=\int_0^\infty (1-F_T(t))(1-F_X^{*\lceil\lambda\rceil}(t))dt=ET\int_0^\infty f_e(t)(1-F_X^{*\lceil\lambda\rceil}(t))dt
\end{equation}
where $f_e(t)=(1-F_T(t))/ET$ and
$F_T$ and $F_X$ are the cumulative distribution functions of $T$ and $X$, respectively. In this case there is no explicit formula for $\lambda_\alpha$, but in many cases it can be computed numerically. A case which is worth pointing out is when $X\sim\exp(\theta)$. In this case it can be easily verified that
\begin{equation}
E\tau_\lambda=ET\sum_{k=0}^{\lceil\lambda\rceil-1}Ee^{-\theta T_e}\frac{(\theta T_e)^k}{k!}
\end{equation}
where $T_e$ has a distribution with density $f_e$. In this case
\begin{equation}Ee^{-\theta T_e}=\frac{1-Ee^{-\theta T}}{\theta ET}\end{equation}
and the $k$th derivative of this function with respect to $\theta$ is given by
\begin{equation}
(-1)^k Ee^{-\theta T_e}T_e^k\ .
\end{equation}
Thus, in principal, the knowledge of $Ee^{-s T}$ for every $s\ge 0$ gives us a procedure for finding everything that is needed in order to compute the optimal $\tau$ in this case.

\section{Separable convex optimization with linear constraints}\label{sec:separable}
Obviously, Corollary~\ref{cor:main2} can be applied to the following optimization problem in which $f_i$ are convex (not necessarily differentiable) functions, $t_i$ nonnegative reals (possibly infinite) and $a_i$ are strictly positive and finite.
\begin{align}
\min\ &\sum_{i=1}^nf_i(x_i)\nonumber\\
\text{s.t.}\ &x_i\in[0,t_i]\cap[0,\infty)\ , \ \forall 1\le i\le n\\
&\sum_{i=1}^na_ix_i=\alpha\nonumber
\end{align}
When $f_i$ are differentiable, $t_i=\infty$, $a_i=1$ and $\alpha=1$, the results are consistent with the famous Gibb's Lemma, noting that if $\xi_i(0)\ge \lambda$ then necessarily $\eta_i(\lambda)=0$. This is a standard convex optimization problem with a separable objective function and linear constraints and the number of references is huge (for the case where $f_i$ are differentiable). For example, quite a few examples are given in \cite{Patriksson2008}. The standard solution (under differentiability assumptions) is by applying the Karush-Kuhn-Tucker conditions or Gibb's Lemma.

\section{Regulation of $M/G/1$ system}\label{sec: regulation}
Consider a single-server first-come-first-served $M/G/1$ queue with an arrival rate $\lambda$ where the service demand distribution is determined endogenously by the following mechanism: Each customer decides during service when to terminate that service and leave the system. This decision is influenced by three factors. Linear waiting time cost (excluding service time),  a marginal utility modelled by a general nonincreasing right-continuous stochastic process which is observed from the moment service begins until departure from the system and a price the customer pays the system which is some function of the time this customer occupies the server. Now, customers must join the queue, so reneging or abandonments are not allowed. Therefore, their decisions concern only the question of when to quit service after it has begun. Hence, for every customer the decision is an optimal stopping problem with respect to the information generated by their marginal utilities. Since the marginal utilities of the customers are iid processes which are independent from the arrival process, then so are the resulting decisions of the customers. Thus, eventually the model under examination is a regular $M/G/1$ queue with service distribution which is determined endogenously by the mechanism above. The question is how to determine a price function which implies optimal resource allocation from a social point of view? Assume that the social optimality criterion is the expected utility of a customer in steady-state for the resulting $M/G/1$ system for certain family of price functions. Now, instead of solving this problem directly, \cite{jacobovic2020} first optimizes the performance measure  over all service distributions. Then, once an optimal service distribution is derived, it turns out to be possible to construct a price function for which the corresponding optimal stopping times of the customers are distributed according to the optimal service distribution which was initially derived. If $V(\cdot)$ is a process with the distribution of the marginal utilities of the customers and we assume that $V(0)$ is a nonnegative random variable such that $EV^2(0)<\infty$, then  an optimal service distribution is a solution of the problem
\begin{equation} \label{optimization: phase I-regulation}
\begin{aligned}
& \max:
& &E\int_0^S\left[ V(s)-s\frac{\lambda}{1- \lambda ES}\right]ds \\
& \ \text{s.t:}
& & S \text{ is a random variable}\,, \\ & & & S\geq0 \ , \ P \text{-a.s.}\,, \\ & & & ES<\lambda\,.
\end{aligned}
\end{equation}
To solve this problem, as in Sections \ref{sec: Levy storage} and \ref{sec:clearing}, \cite{jacobovic2020} includes a two-phase method. Given some $\alpha\in\left[0,\lambda^{-1}\right)$, Phase I solves \eqref{optimization: phase I-regulation} with an additional constraint $ES=\alpha$, \textit{\textit{i.e}}.,
\begin{equation} \begin{aligned}
& \max:
& &E\int_0^S\left[ V(s)-s\frac{\lambda}{1- \lambda \alpha}\right]ds \\
& \ \text{s.t:}
& & S \text{ is a random variable}\,, \\ & & & S\geq0 \ , \ P \text{-a.s.}\,, \\ & & & ES=\alpha\,.
\end{aligned}
\end{equation}
This is a special case of the optimization solved in Section \ref{sec:main}.


\begin{thebibliography}{99}
\bibitem{Bertsekas1973}
Bertsekas, D. P. (1973). Stochastic optimization problems with nondifferentiable cost functionals. \textit{Journal of Optimization Theory and Applications}, \textbf{12}, 218-231.

\bibitem{Cvitanic2001}
Cvitanic, J. and I. Karatzas. (2001). Generalized Neyman-Pearson lemma via convex duality. \textit{Bernoulli}, \textbf{7}, 79-97.

\bibitem{Ermoliev1983}
Ermoliev, Y. (1983). Stochastic quasigradient methods and their application to system optimization. \textit{Stochastics: An International Journal of Probability and Stochastic Processes}, \textbf{9}, 1-36.

\bibitem{Fishburn1998}
Fishburn, P. C. (1998). Stochastic utility. \textit{Handbook of utility theory}, 1, 273-319.


\bibitem{Follmer2011}
F\"ollmer, H. and A. Schied. (2011). \textit{Stochastic finance: an introduction in discrete time}. Walter de Gruyter.

\bibitem{Goffin2002}
Goffin, J. L. and J. P. Vial. (2002). Convex nondifferentiable optimization: A survey focused on the analytic center cutting plane method. \textit{Optimization methods and software}, \textbf{17}, 805-867.

\bibitem{HKL80} Helgason, R., Kennington and H. Hall. (1980). A polynomially bounded algorithm for a singly constrained quadratic program. {\em Mathematical programming}, {\bf 18}, 388-343.

\bibitem{jacobovic2020}
Jacobovic, R. (2020). Regulation of M/G/1 queue with customers having nonincreasing stochastic marginal utilities. \textit{arXiv preprint arXiv:2001.07664}.

\bibitem{JK2019}
Jacobovic, R., and Kella, O. (2019). Asymptotic independence of regenerative processes with a special dependence structure. \textit{Queueing Systems}, \textbf{93}, 139-152.

\bibitem{JK19} Jacobovic, R., and Kella, O. (2019). Steady-state optimization of an exhaustive L\'evy storage process with intermittent output and random output rate. \textit{arXiv preprint arXiv:1906.10621}.

\bibitem{Jihan2017}
Jian, X., Yi, F. and J. Zhang. (2017). Investment and consumption problem in finite time with consumption constraint. \textit{ESAIM: Control, Optimisation and Calculus of Variations}, \textbf{23}, 1601-1615.

\bibitem{JXZ08} Jin, H., Xu, Z. Q. and X. Y. Zhou. (2008). A convex stochastic optimization problem arising from portfolio selection. {\em Mathematical Finance}, {\bf 18}, 171-183.

\bibitem{KS88} Karatzas, I., and Shreve, S. E. (1988). {\em Brownian Motion and Stochastic Calculus.} Springer.

\bibitem{KS1998}
Karatzas, I., and Shreve, S. E. (1998). \textit{Methods of mathematical finance}. New York: Springer.

\bibitem{kella98} Kella, O. (1998). An exhaustive L\'evy storage process with intermittent output. {\em Stochastic Models}, {\bf 14}, 979-992.

\bibitem{kellastadje15} Kella, O. and W. Stadje. (2015). A clearing system with impatient passengers: asymptotics and estimation in a bus stop model. \textit{ Queueing Systems}, {\bf 80}, 1-14.

\bibitem{kellayor10} Kella, O. and M. Yor. (2010). A new formula for some linear stochastic equations with applications. \textit{Ann. Appl. Probab.}, {\bf 20}, 367-381.

\bibitem{Koo2016}
Koo, J. L., Ahn, S. R., Koo, B. L., Koo, H. K. and Y. H. Shin. (2016). Optimal consumption and portfolio selection with quadratic utility and a subsistence consumption constraint. \textit{Stochastic Analysis and Applications}, \textbf{34}, 165-177.

\bibitem{Korn1997}
Korn, R. (1997). \textit{Optimal portfolios: stochastic models for optimal investment and risk management in continuous time}. World Scientific.

\bibitem{Korn2005}
Korn, R. (2005). Optimal portfolios with a positive lower bound on final wealth. \textit{Quantitative finance}, \textbf{5}, 315-321.

\bibitem{Lanker2006}
Lakner, P. and L. Ma Nygren. (2006). Portfolio optimization with downside constraints. \textit{Mathematical Finance: An International Journal of Mathematics, Statistics and Financial Economics}, \textbf{16}, 283-299.

\bibitem{Lemarechal1989}
Lemar\'echal, C. (1989). Chapter vii, Nondifferentiable Optimization. \textit{Handbooks in operations research and management science}, \textbf{1}, 529-572.

\bibitem{Patriksson2008}
Patriksson, M. (2008). A survey on the continuous nonlinear resource allocation problem. \textit{European Journal of Operational Research}, \textbf{185}, 1-46.

\bibitem{Plambeck1996}
Plambeck, E. L., Fu, B. R., Robinson, S. M. and R. Suri. (1996). Sample-path optimization of convex stochastic performance functions. \textit{Mathematical Programming}, \textbf{75}, 137-176.

\bibitem{prekopa2013}
Pr\'ekopa, A. (2013). \textit{Stochastic programming} (Vol. 324). Springer Science \& Business Media.

\bibitem{Rusc2003}
Ruszczy\'nski, A. and A. Shapiro. (2003). Stochastic programming models. \textit{Handbooks in operations research and management science}, \textbf{10}, 1-64.

\bibitem{Serfozo1978}
Serfozo, R., and Stidham, S. (1978). Semi-stationary clearing processes. \textit{Stochastic Processes and their Applications}, \textbf{6}, 165-178.

\bibitem{Sigman1993}
Sigman, K., and Wolff, R. W. (1993). A review of regenerative processes. \textit{SIAM review}, \textbf{35}, 269-288.

\bibitem{stidham73} S. Stidham Jr. (1976). Stochastic clearing systems. {\em Stoch. Proc. Appl.}, {\bf 2}, 85-113.

\bibitem{stidham77} S. Stidham Jr. (1977). Cost models for stochastic clearing systems. {\em Oper. Res.}, {\bf 25}, 100-127.

\bibitem{whitt81} Whitt, W. (1981). The stationary distribution of a stochastic clearing process. \textit{Operations Research}, \textbf{29}, 294-308.

\bibitem{Zornig2014}
Z$\ddot{\text{o}}$rnig, P. (2014). \textit{Nonlinear Programming: An Introduction}. Walter de Gruyter.

\end{thebibliography}
\end{document}